\documentclass[12pt]{article}
\usepackage{lmodern}
\usepackage{microtype}
\usepackage[bookmarks=false]{hyperref}

\usepackage[english]{babel}
\usepackage[all]{xy}
\usepackage{fancyhdr}
\usepackage{setspace, amsmath, amsthm, amssymb, amsfonts, amscd, epic, graphicx, ulem, dsfont}
\usepackage{amsthm}
\usepackage[T1]{fontenc}
\usepackage{cases}
\newtheorem{theorem}{Theorem}

\newtheorem{corollary}[theorem]{Corollary}
\newtheorem{example}[theorem]{Example}

\newtheorem{remark}[theorem]{Remark}

\makeatletter

\@addtoreset{equation}{section}
\makeatother

\title{On generalized $\xi$-Parallel Maps}

\author{Ahmed Mohammed Cherif\footnote{University Mustapha Stambouli Mascara, Faculty of Exact Sciences, Mascara 29000, Algeria. Email: a.mohammedcherif@univ-mascara.dz}}
\date{}

\begin{document}
\maketitle
	
\begin{abstract}
In this paper, we introduce and study the notion of generalized $\xi$-parallel maps between Riemannian manifolds, where $\xi$ is a vector field on the target manifold. We define the associated energy functional and derive its first variation formula, leading to the Euler--Lagrange equation characterizing generalized $\xi$-parallel maps. We establish fundamental properties of this new class of maps and investigate its relationship with harmonic and biharmonic maps. Furthermore, we study generalized $\xi$-parallel curves in space forms and examine generalized $\xi$-parallel submanifolds of space forms admitting concircular vector fields.
\begin{flushleft}
\end{flushleft}
\end{abstract}

\section{Introduction}
 The following definition extends the classical notion of a vector field parallel along a curve to the more general setting of smooth maps between Riemannian manifolds \cite{ON}. Let $\varphi:(M,g)\to(N,h)$ be a smooth map between Riemannian manifolds. A section $V \in \Gamma(\varphi^{-1}TN)$ is called parallel along $\varphi$ if
\begin{equation}\label{eq1.1}
\nabla^{\varphi}_X V = 0, \quad \forall X \in \Gamma(TM),
\end{equation}
where $\nabla^\varphi$ denotes the pullback connection induced by the Levi-Civita connection on $(N,h)$ \cite{baird,ES,ON,YX}. \\
Let $\xi\in\Gamma(TN)$ be a nonzero vector field. A map $\varphi:(M,g)\to(N,h)$ is called $\xi$-parallel if the section $\xi \circ \varphi \in \Gamma(\varphi^{-1}TN)$ is parallel along $\varphi$. Note that every smooth map $\varphi:(M,g)\to\mathbb{R}^n$ is $\xi$-parallel, since $\mathbb{R}^n$ admits nonzero parallel vector fields with respect to the Euclidean connection. On the other hand, the identity map on the unit sphere equipped with its standard metric is not $\xi$-parallel, since the sphere admits no nonzero parallel vector fields. The above examples show that there exist both $\xi$-parallel and non $\xi$-parallel maps. In particular, the existence of $\xi$-parallel maps depends on the geometry of the target manifold $(N,h)$. Therefore, the study of this class of maps provides valuable information about both the map itself and the geometry of its codomain.\\
We now generalize the above definition by introducing the following $\xi$-energy functional
\begin{equation}\label{eq1.2}
E_\xi(\varphi;D)=\frac{1}{2}\int_{D}\big|\nabla^\varphi(\xi\circ\varphi)\big|^2\,v_g,
\end{equation}
where $D\subset M$ is an arbitrary compact domain, $\xi$ is a fixed vector field on $N$, and $v_g$ denotes the Riemannian volume element on $(M,g)$.\\
Observe that, with respect to any local orthonormal frame $\{e_i\}_{i=1}^m$ on $(M,g)$, the $\xi$-energy density of $\varphi$ is given by
\begin{equation}\label{eq1.3}
\big|\nabla^\varphi(\xi\circ\varphi)\big|^2
=\sum_{i=1}^m h\big(\nabla^\varphi_{e_i}(\xi\circ\varphi),\nabla^\varphi_{e_i}(\xi\circ\varphi)\big).
\end{equation}
A smooth map $\varphi:(M,g)\to(N,h)$ between Riemannian manifolds is called a generalized $\xi$-parallel map if it is a critical point of the $\xi$-energy
functional for some nonzero vector field $\xi\in\Gamma(TN)$. \\
The main purpose of this paper is to introduce and investigate the notion of generalized $\xi$-parallel maps between Riemannian manifolds and to study its geometric implications for curves and submanifolds. After defining the associated energy functional, we derive its first variation formula and obtain the corresponding Euler--Lagrange equation, which characterizes generalized $\xi$-parallel maps. We then establish several fundamental properties of this new class of maps and examine their relationship with the classical notion of $\xi$-parallel maps.\\
We show that, when $\xi$ is a concurrent vector field \cite{KR,YN}, the notion of generalized $\xi$-parallel maps reduces to that of harmonic maps, providing a natural generalization of harmonic map theory \cite{baird,ES,YX}.
In particular, under suitable assumptions on the vector field $\xi$, we construct examples of generalized $\xi$-parallel maps that are not $\xi$-parallel. We also investigate generalized $\xi$-parallel curves in space forms and characterize generalized $\xi$-parallel submanifolds in space forms admitting a concircular vector \cite{KR,YN,YN2} field.


\section{Generalized $\xi$-parallel maps}

The following theorem establishes the first variation formula for the
energy functional $E_\xi$ and introduces the associated $\xi$-tension field
of generalized $\xi$-parallel maps, denoted by $\tau_{\xi}(\varphi)$. In particular,
a smooth map $\varphi$ is a critical point of $E_\xi$ if and only if $\tau_{\xi}(\varphi)=0.$

\begin{theorem}\label{th2.1}
Let $\varphi:(M,g)\to(N,h)$ be a smooth map between Riemannian manifolds, and let
$\{\varphi_t\}_{t\in(-\epsilon,\epsilon)}$ be a smooth variation of $\varphi$ supported in a compact domain $D\subset M$. Then
\begin{equation}\label{eq2.1}
\frac{d}{dt}E_{\xi}(\varphi_t;D)\Big|_{t=0}
= -\int_{D} h\big(v,\tau_{\xi}(\varphi)\big)\, v_g,
\end{equation}
where $v=\big.\frac{d\varphi_t}{dt}\big|_{t=0}$ denotes the variation vector field of $\{\varphi_t\}_{t\in(-\epsilon,\epsilon)}$, $\tau_{\xi}(\varphi)$ is the $\xi$-tension field of $\varphi$ given by
\begin{equation}\label{eq2.2}
\tau_{\xi}(\varphi)
= \operatorname{Tr} R^{N}\big(\xi\circ\varphi,\sigma \circ d\varphi\big)d\varphi
+ \sigma^{*}\big(\operatorname{Tr}\, (\nabla^\varphi)^2 (\xi\circ\varphi)\big),
\end{equation}
$\sigma:\Gamma(TN)\to\Gamma(TN)$ is defined by $\sigma(X)=\nabla^N_X\xi$, and $\sigma^{*}$ denotes the adjoint of $\sigma$.
Here, $\nabla^N$ denotes the Levi--Civita connection, and $R^N$ denotes the Riemannian curvature of $(N,h)$.
\end{theorem}

\begin{proof}
Let $\phi$ be a smooth map defined by
\begin{eqnarray*}
   \phi:(-\epsilon,\epsilon)\times M&\rightarrow& N. \\
  (t,x)\quad &\mapsto&\phi(t,x)=\varphi_{t}(x)
\end{eqnarray*}
Let $\{e_{i}\}_{i=1}^m$ be an orthonormal frame with respect to $g$ on $M$ such that $\nabla_{e_{j}}^{M}e_{j}=0$ at $x\in M$ for all $i,j=1,\ldots,m$. Observe that, $\phi(x,0)=\varphi(x)$, and the variation vector field $v$ associated to  $\{\varphi_{t}\}_{t\in(-\epsilon,\epsilon)}$ is given by $v=d\phi(\frac{\partial}{\partial t})|_{t=0}$. We have
\begin{equation}\label{eq2.3}
  \frac{d}{dt}E_{\xi}(\varphi_{t};D)\Big|_{t=0}=\frac{1}{2}\int_{D}\frac{\partial}{\partial t}\big|\nabla^{\varphi_t}(\xi\circ\varphi_t)\big|^2\Big|_{t=0}\,v_{g}.
\end{equation}
We compute the following term
\begin{eqnarray}\label{eq2.4}
 \frac{1}{2}\frac{\partial}{\partial t}\big|\nabla^{\varphi_t}(\xi\circ\varphi_t)\big|^2
 &=&\nonumber\frac{1}{2}\sum_{i=1}^m\frac{\partial}{\partial t} h\big(\nabla^{\varphi_t}_{e_i}(\xi\circ\varphi_t),\nabla^{\varphi_t}_{e_i}(\xi\circ\varphi_t)\big)\\
 &=&\nonumber\frac{1}{2}\sum_{i=1}^m\frac{\partial}{\partial t} h\big(\nabla^\phi_{e_i}(\xi\circ\phi),\nabla^\phi_{e_i}(\xi\circ\phi)\big)\\
 &=&\sum_{i=1}^m h\big(\nabla^\phi_{\frac{\partial}{\partial t}}\nabla^\phi_{e_i}(\xi\circ\phi),\nabla^\phi_{e_i}(\xi\circ\phi)\big).
\end{eqnarray}
Since $[\frac{\partial}{\partial t},e_i]=0$ for all $i=1,\ldots,m$, it follows from the definition of the curvature tensor of $(N,h)$ \cite{ON}, that
\begin{eqnarray}\label{eq2.5}
 \frac{1}{2}\frac{\partial}{\partial t}\big|\nabla^{\varphi_t}(\xi\circ\varphi_t)\big|^2
 &=&\nonumber\sum_{i=1}^m h\big(R^N(d\phi(\frac{\partial}{\partial t}),d\phi(e_i))(\xi\circ\phi),\nabla^\phi_{e_i}(\xi\circ\phi)\big)\\
 & &+\sum_{i=1}^m h\big(\nabla^\phi_{e_i}\nabla^\phi_{\frac{\partial}{\partial t}}(\xi\circ\phi),\nabla^\phi_{e_i}(\xi\circ\phi)\big).
\end{eqnarray}
Define $\omega\in\Gamma(T^*M)$ by $\omega(X)=h\big(\nabla^N_v\xi,\nabla^\varphi_{X}(\xi\circ\varphi)\big)$. Using (\ref{eq2.5}), we get at $x$
\begin{eqnarray}\label{eq2.6}
 \frac{1}{2}\frac{\partial}{\partial t}\big|\nabla^{\varphi_t}(\xi\circ\varphi_t)\big|^2\Big|_{t=0}
 &=&\nonumber\sum_{i=1}^m h\big(R^N(v,d\varphi(e_i))(\xi\circ\varphi),\nabla^\varphi_{e_i}(\xi\circ\varphi)\big)\\
 & &+\operatorname{div}\omega-\sum_{i=1}^m h\big(\nabla^N_{v}\xi,\nabla^\varphi_{e_i}\nabla^\varphi_{e_i}(\xi\circ\varphi)\big).
\end{eqnarray}
By substituting (\ref{eq2.6}) into (\ref{eq2.3}), together with $\sigma(d\varphi(e_i))=\nabla^\varphi_{e_i}(\xi\circ\varphi)$ for
$i=1,\ldots,m$ and $\nabla^N_{v}\xi=\sigma(v)$, and using the definition of the adjoint operator
$\sigma^*$ of $\sigma$, as well as applying the divergence theorem
\cite{baird}, we obtain (\ref{eq2.1}) of Theorem \ref{th2.1}.
\end{proof}

According to Theorem \ref{th2.1}, we deduce the following Corollary.

\begin{corollary}\label{co2.1}
A smooth map $\varphi:(M,g)\to(N,h)$ is generalized $\xi$-parallel if and only if $\tau_{\xi}(\varphi)=0$.
\end{corollary}

\begin{remark}
Equation $\tau_{\xi}(\varphi)=0$ is called the generalized $\xi$-parallel equation,
and in local coordinates, it becomes a nonlinear system of PDEs.
We consider $(U,x_i)$ and $(V,y_\alpha)$ two local coordinate charts at $x \in M$ and
$\varphi(x)\in N$ respectively, and denote by ${}^{M}\Gamma^{k}_{ij}$ and
${}^{N}\Gamma^{\alpha}_{\beta\gamma}$ the Christoffel symbols of the
Levi-Civita connection of $M$ and $N$ respectively, then $\varphi$
is generalized $\xi$-parallel  if and only if
\begin{eqnarray*}
   0&=& g^{ij}\big\{ \varphi^\alpha_i \varphi^\beta_j \big( \xi^a \xi^b_\alpha\, {}^NR_{ab\beta}^l+\xi^a\xi^b \,{}^N\Gamma_{\alpha b}^c \, {}^NR_{ac\beta}^l\big)h_{l\theta}+\big[\varphi^\beta_{ij}\big(\xi^l_\beta+\xi^a \,^N\Gamma_{\beta a}^l\big)\\
    & & +\varphi^\alpha_i \varphi^\beta_j \big(\xi^l_\beta+\xi^a \,^N\Gamma_{\beta a}^l\big)_\alpha+\varphi^\alpha_i \varphi^\beta_j \big(\xi^c_\beta+\xi^a \,^N\Gamma_{\beta a}^c\big) {}^N\Gamma_{\alpha c}^l\\
    & &
        -{}^M\Gamma_{ij}^k\varphi^\alpha_k\big(\xi^l_\alpha+\xi^a \,^N\Gamma_{\alpha a}^l\big)\big]\big(\xi^t_\theta+\xi^s \,{}^N\Gamma_{\theta s}^t\big) h_{tl}\big\}\circ\varphi,
\end{eqnarray*}
for all $\theta=1,\ldots,n$, where $\varphi^\alpha_i=\partial \varphi^\alpha/\partial x_{i}$, and
$\xi^{b}_{a}=\partial \xi^{b}/\partial y^{a}$,
Here we use the Einstein summation convention for
$i,j,k=1,\ldots,m$ and
$a,b,c,l,s,t,\alpha,\beta=1,\ldots,n$,
where $m=\dim M$ and $n=\dim N$.
\end{remark}

Using (\ref{eq2.2}) of Theorem \ref{th2.1}, and Corollary \ref{co2.1}, we obtain.

\begin{corollary}\label{co2.2}
Every $\xi$-parallel map $\varphi:(M,g)\to(N,h)$ is a generalized $\xi$-parallel map.
\end{corollary}

\begin{remark}
Let $\varphi:(M,g)\to\mathbb{R}^n$ be a smooth map and $\xi\in\Gamma(T\mathbb{R}^n)$ such that
\[
\nabla^\varphi_X\nabla^\varphi_X(\xi\circ\varphi)
-\nabla^\varphi_{\nabla^M_X X}(\xi\circ\varphi)=0,
\]
for every $X\in\Gamma(TM)$. Then, $\varphi$ is generalized $\xi$-parallel map.
\end{remark}

The following theorem establishes a direct relationship between generalized $\xi$-parallel maps with respect to a concurrent vector field $\xi$ and harmonic maps \cite{ES}. It shows that, when the target manifold admits a concurrent vector field \cite{F,YN2}, these two notions coincide.

\begin{theorem}\label{th-concurrent}
Let $(M,g)$ be a Riemannian manifold and let $(N,h)$ be a Riemannian manifold admitting a concurrent vector field $\xi$, that is,
\[
\nabla^{N}_{X}\xi=X,\quad \forall\,X\in\Gamma(TN).
\]
Then, a smooth map $\varphi:(M,g)\to(N,h)$ is generalized $\xi$-parallel if and only if $\varphi$ is harmonic.
\end{theorem}

From Theorem \ref{th-concurrent}, we obtain the following result.

\begin{corollary}
Every harmonic map $\varphi:(M,g)\to\mathbb{R}^n$ is a generalized $\xi$-parallel map with respect to the position vector field $\xi$.
\end{corollary}

This corollary shows that there exist proper generalized $\xi$-parallel maps (i.e., generalized $\xi$-parallel non $\xi$-parallel maps) into Euclidean space.

\begin{remark}\label{remark}
Let $(M,g)$ be a Riemannian manifold and let $(N,h)$ be a Riemannian manifold admitting a concircular vector field $\xi$ with concircular function $f$ on $N$, that is,
\begin{equation}\label{eq-concircular}
\nabla^{N}_{X}\xi=fX,\quad \forall\,X\in\Gamma(TN),
\end{equation}
for some smooth function $f$ on $N$ \cite{F,KR,YN,YN2}.
\begin{enumerate}
\item If $f>0$. Then, every generalized $\xi$-parallel map from $(M,g)$ to $(N,h)$ is an $f^2$-harmonic map \cite{DMZO}. Moreover, a smooth map $\varphi:(M,g)\to(N,h)$ is generalized $\xi$-parallel if and only if it satisfies
\[
(f\circ\varphi)^2 \tau(\varphi)
+ d\varphi\big(\operatorname{grad}^M (f\circ\varphi)^2\big)
-\frac{1}{2}|d\varphi|^2\,(\operatorname{grad}^N f^2)\circ\varphi
=0,
\]
where $\tau(\varphi)$ denotes the tension field of $\varphi$ \cite{ES}. In particular, generalized $\xi$-parallel maps generalize both harmonic maps and $f^2$-harmonic maps.
\item If $f\neq0$. A smooth map $\varphi$ from $(M,g)$ to $(N,h)$ is generalized $\xi$-parallel if and only if
\begin{equation*}
\operatorname{Tr} R^{N}\big(\xi\circ\varphi, d\varphi\big)d\varphi
+ \operatorname{Tr}\, (\nabla^\varphi)^2 (\xi\circ\varphi)=0.
\end{equation*}
 Moreover, if $\tau(\varphi)=C\,(\xi\circ\varphi)$ for some nonzero constant $C$, then the map $\varphi:(M,g)\to(N,h)$ is generalized $\xi$-parallel if and only if it is a biharmonic map \cite{Jiang}. Therefore, the notion of generalized $\xi$-parallel maps provides a natural generalization of biharmonic maps.
\end{enumerate}
\end{remark}

\section{Generalized $\xi$-parallel curves in space forms}

Let $(N^3(c),h)$ be a three-dimensional Riemannian manifold of constant curvature $c$, and let
$\gamma: I \longrightarrow (N^3(c),h)$ be a curve parametrized by arc length, where
$I \subseteq \mathbb{R}$ is an open interval. Consider an orthonormal frame field
$\{T, N, B\}$ along $\gamma$, with $T = d\gamma(\frac{d}{dt})$ the unit tangent vector,
$N$ the unit normal vector in the direction of $\nabla_T T$, and $B$ chosen so that
$\{T, N, B\}$ forms a positively oriented basis. Denote by $\nabla$ the Levi-Civita
connection of $(N^3(c),h)$. Then, the Frenet equations along $\gamma$ are given by
\[
\left\{
\begin{aligned}
\nabla_T T &= k N,\\
\nabla_T N &= -k T + \tau B,\\
\nabla_T B &= -\tau N,
\end{aligned}
\right.
\]
where $k$ and $\tau$ are the geodesic curvature and geodesic torsion of $\gamma$, respectively.
The following theorem provides a characterization of generalized $\xi$-parallel curves in space forms with respect to the vector field $\xi$ on $N^3(c)$, which is written along $\gamma$ as $\xi=\xi_1 T+\xi_2 N+\xi_3 B.$
We denote by $\sigma_{ij}$ the components of $\sigma$ with respect to the frame $\{T,N,B\}$, that is,
$\sigma_{ij}=h( X_i,\sigma(X_j)),$ where $X_1=T$, $X_2=N$, and $X_3=B$.

\begin{theorem}\label{th5}
A curve $\gamma:I\longrightarrow (N^3(c),h)$ is generalized $\xi$-parallel if and only if
\[
\left\{
\begin{array}{l}
\displaystyle
(\xi_1'-k\xi_2)^2+(\xi_2'+k\xi_1-\tau\xi_3)^2+(\xi_3'+\tau\xi_2)^2=C^{st},\\[2ex]
\displaystyle
c\,(\xi_1'\xi_2-k\xi_2^2-\xi_1\xi_2'-k\xi_1^2+\tau\xi_1\xi_3)+(\alpha'-k\beta)\sigma_{12}\\[1.5ex]
\displaystyle
+(k\alpha+\beta'-\tau\delta)\sigma_{22}+(\tau\beta+\delta')\sigma_{32}=0,\\[2ex]
\displaystyle
c\,(\xi_1'\xi_3-k\xi_2\xi_3-\xi_1\xi_3'-\tau\xi_1\xi_2)+(\alpha'-k\beta)\sigma_{13}\\[1.5ex]
\displaystyle
+(k\alpha+\beta'-\tau\delta)\sigma_{23}+(\tau\beta+\delta')\sigma_{33}=0.
\end{array}
\right.
\]
\end{theorem}

\begin{proof}
Using $T=d\gamma(\frac{d}{dt})$, the $\xi$-tension field of $\gamma$ is given by
\begin{equation}\label{eq5.2}
\tau_{\xi}(\gamma)
= R^{N}\big(\xi\circ\gamma,\sigma(T)\big)T
+ \sigma^{*}\big(\nabla^\gamma_{\frac{d}{dt}} \nabla^\gamma_{\frac{d}{dt}}(\xi\circ\gamma)\big),
\end{equation}
where $R^N$ denotes the Riemannian curvature tensor of $(N^3(c),h)$.
We compute the first term of (\ref{eq5.2}). As $R^{N}(X,Y)Z=c\,h( Y,Z) X-c\,h( X,Z) Y$, we obtain
\begin{eqnarray}\label{eq5.3}
R^{N}\big(\xi\circ\gamma,\sigma(T)\big)T
&=&\nonumber c\,h(\sigma(T),T)(\xi\circ\gamma)-c\,h(\xi\circ\gamma,T)\sigma(T)\\
&=& c\,h(\nabla^N_T\xi,T)(\xi\circ\gamma)-c\,h(\xi\circ\gamma,T)\nabla^N_T\xi.
\end{eqnarray}
Take $\xi=\xi_1 T+\xi_2 N+\xi_3 B$. Substituting this into (\ref{eq5.3}) and using the Frenet equations, we find that
\begin{eqnarray}\label{eq5.4}
R^{N}\big(\xi\circ\gamma,\sigma(T)\big)T
&=&\nonumber c\,(\xi_1'-k\xi_2)(\xi_1 T+\xi_2 N+\xi_3 B)
    -c\,\xi_1\big(\xi_1'T+\xi_2'N\\
& &\nonumber+\xi_3'B+k\xi_1N-k\xi_2 T+\tau\xi_2B-\tau \xi_3 N\big)\\
&=&\nonumber c\big(\xi_1'\xi_2-k\xi_2^2-\xi_1\xi_2'-k\xi_1^2+\tau\xi_1\xi_3\big)N\\
& & +c\big(\xi_1'\xi_3-k\xi_2\xi_3-\xi_1\xi_3'-\tau\xi_1\xi_2\big)B.
\end{eqnarray}
Now, we compute the second term in (\ref{eq5.2}). Applying the Frenet equations, we obtain the following
\begin{equation}\label{eq5.5}
\sigma(T)=\nabla^\gamma_{\frac{d}{dt}}(\xi\circ\gamma)=\alpha T+\beta N+\delta B,
\end{equation}
where $\alpha=\xi_1'-k\xi_2$, $\beta=\xi_2'+k\xi_1-\tau\xi_3$, and $\delta=\xi_3'+\tau\xi_2$. It follows that
\begin{eqnarray}\label{eq5.6}
 \nabla^\gamma_{\frac{d}{dt}} \nabla^\gamma_{\frac{d}{dt}}(\xi\circ\gamma)
 &=& (\alpha'-k\beta)T+(k\alpha+\beta'-\tau\delta)N+(\tau\beta+\delta')B.
\end{eqnarray}
Consequently, using (\ref{eq5.5}) and (\ref{eq5.6}), together with the relation
$\sigma_{ij}=h( X_i,\sigma(X_j)),$ where $X_1=T$, $X_2=N$, and $X_3=B$, we obtain
\begin{eqnarray*}
 h\big(\sigma^*\big(\nabla^\gamma_{\frac{d}{dt}} \nabla^\gamma_{\frac{d}{dt}}(\xi\circ\gamma)\big),T\big)
 &=& \alpha\alpha'+\beta\beta'+\delta\delta',\\
  h\big(\sigma^*\big(\nabla^\gamma_{\frac{d}{dt}} \nabla^\gamma_{\frac{d}{dt}}(\xi\circ\gamma)\big),N\big)
 &=& (\alpha'-k\beta)\sigma_{12}+(k\alpha+\beta'-\tau\delta)\sigma_{22}+(\tau\beta+\delta')\sigma_{32},\\
  h\big(\sigma^*\big(\nabla^\gamma_{\frac{d}{dt}} \nabla^\gamma_{\frac{d}{dt}}(\xi\circ\gamma)\big),B\big)
 &=& (\alpha'-k\beta)\sigma_{13}+(k\alpha+\beta'-\tau\delta)\sigma_{23}+(\tau\beta+\delta')\sigma_{33}.
\end{eqnarray*}
Substituting these formulas and (\ref{eq5.4}) into (\ref{eq5.2}), we deduce the system stated in Theorem \ref{th5}.
\end{proof}

\begin{remark}
The first equation of the system in Theorem \ref{th5} shows that
$$\sigma\big(d\gamma(\frac{d}{dt})\big)=\nabla^\gamma_{\frac{d}{dt}}(\xi\circ\gamma),$$ has a constant norm along the curve $\gamma$.
\end{remark}

In the particular case where $N^3(c)$ is the Euclidean $3$-space $(c=0)$, Theorem \ref{th5} yields the following example.

\begin{example}
We consider the helix $\gamma(t)=(\cos t,\sin t,t)$ in $\mathbb{R}^3$. Note that the Frenet frame is given by
\begin{eqnarray*}
  T &=& \frac{1}{\sqrt{2}}(-\sin t,\cos t,1); \\
  N &=& (-\cos t,-\sin t,0); \\
  B &=& \frac{1}{\sqrt{2}}(\sin t,-\cos t,1).
\end{eqnarray*}
Let $\xi=f(z)\,\partial_z$ be a vector field on $\mathbb{R}^3$, where $f$ is a smooth real-valued function of $z$.
Observe that, along the curve $\gamma$, we have $\xi\circ\gamma=\frac{f}{\sqrt{2}}(T+B)$. Therefore, $\xi_1=\frac{f}{\sqrt{2}}$, $\xi_2=0$, and $\xi_3=\frac{f}{\sqrt{2}}$. Moreover,
\begin{eqnarray*}
  \sigma(T)=\sigma(B)= \frac{f'}{2}( T+ B),\quad \sigma(N) = 0.
\end{eqnarray*}
Since the helix $\gamma$ has constant curvature and torsion $k=\tau=\frac{1}{2}$, Theorem \ref{th5} implies that $\gamma$ is proper generalized $\xi$-parallel if and only if $f(z)=\alpha z+\beta$, where $\alpha,\beta\in\mathbb{R}$ and $\alpha\neq 0$.
\end{example}

\section{Generalized $\xi$-parallel submanifolds in space forms admitting concircular vector fields}

Let $M$ be a submanifold of a Riemannian manifold $(N,h)$, and let
$\mathbf{i}:(M,g)\hookrightarrow (N,h)$ denote the canonical inclusion, where $g=\mathbf{i}^{*}h$ is the induced Riemannian metric on $M$ \cite{baird,ON}.
A submanifold $(M,g)$ is said to be generalized $\xi$-parallel, with respect to a vector field $\xi\in\Gamma(TN)$, if
$\tau_{\xi}(\mathbf{i})=0.$ If, in addition, $\tau(\mathbf{i})\neq 0$, then $(M,g)$ is called a proper generalized $\xi$-parallel submanifold.\\
Throughout the paper, $\nabla^{N}$ and $\nabla^{M}$ denote the Levi--Civita connections on $(N,h)$ and $(M,g)$, respectively, while $\operatorname{grad}^{N}$ and $\operatorname{grad}^{M}$ denote the corresponding gradient operators. Moreover, $B$ denotes the second fundamental form, and $H=(1/m)\operatorname{Tr}B$ the mean curvature vector field of $(M,g)$. With this notation, we establish the following results.

\begin{theorem}\label{th-submanifold}
Let $(M,g)$ be an $m$-dimensional submanifold of a space form $(N(c),h)$ admitting a concircular vector field $\xi$ with nonzero concircular function $f$. Then, $(M,g)$ is a generalized $\xi$-parallel submanifold if and only if
\[
\left\{
\begin{array}{l}
\displaystyle
(m-2)\operatorname{grad}^{M} f = 0,\\[2ex]
\displaystyle
fH -(\operatorname{grad}^{N} f)^\bot = 0.
\end{array}
\right.
\]
Moreover, if $m\neq2$, then $(M,g)$ is generalized $\xi$-parallel submanifold if and only if the function $f$ is constant on $M$, and
$fH=\operatorname{grad}^{N} f$ along $M$.
\end{theorem}

\begin{proof}
Fix a point $x\in M$, and let $\{e_{i}\}_{i=1}^m$ be an orthonormal frame with respect to $g$ on $M$ such that $\nabla_{e_{i}}^{M}e_{j}=0$ at $x\in M$ for all $i,j=1, ...,m$. At $x$, the inclusion $\mathbf{i}:(M,g)\hookrightarrow (N,h)$ is generalized $\xi$-parallel if and only if
\begin{eqnarray}\label{eq4.1}
\sum_{i=1}^m \big[R^{N}\big(\xi\circ\mathbf{i},\sigma(d\mathbf{i}(e_i))\big)d\mathbf{i}(e_i)
+  \sigma^{*}\big(\nabla^\mathbf{i}_{e_i}\nabla^\mathbf{i}_{e_i}(\xi\circ\mathbf{i})\big)\big]=0.
\end{eqnarray}
Since the sectional curvature of $(N,h)$ is a constant $c$,
the first term on the left-hand side of (\ref{eq4.1}) is given at $x$ by
\begin{eqnarray}\label{eq4.2}
\sum_{i=1}^m R^{N}\big(\xi\circ\mathbf{i},\sigma(d\mathbf{i}(e_i))\big)d\mathbf{i}(e_i)
&=&\nonumber c\sum_{i=1}^m \big[h( \sigma(e_i),e_i) \xi-h( \xi,e_i) \sigma(e_i)\big]\\
&=&\nonumber cf\sum_{i=1}^m \big[h( e_i,e_i) \xi-h( \xi,e_i) e_i\big]\\
&=&\nonumber cf \big[m\,\xi-\xi^\top\big]\\
&=& cf \big[(m-1)\xi^\top+m\,\xi^\bot\big].
\end{eqnarray}
We compute the second term on the left-hand side of (\ref{eq4.1}), at $x$ we have
\begin{eqnarray}\label{eq4.3}
\sum_{i=1}^m  \sigma^{*}\big(\nabla^\mathbf{i}_{e_i}\nabla^\mathbf{i}_{e_i}(\xi\circ\mathbf{i})\big)
&=&\nonumber f \sum_{i=1}^m \nabla^N_{e_i}\nabla^N_{e_i}\xi\\
&=&\nonumber f \sum_{i=1}^m \nabla^N_{e_i}(f e_i)\\
&=&\nonumber f \sum_{i=1}^m \big[e_i(f) e_i+f\nabla^N_{e_i} e_i\big]\\
&=&\nonumber f \operatorname{grad}^M f+f^2\sum_{i=1}^m B(e_i,e_i)\\
&=& f \operatorname{grad}^M f+mf^2H.
\end{eqnarray}
Note that, by using the properties $R^{N}(X,Y)\xi=c\,h( Y,\xi) X-c\,h( X,\xi) Y$
and $R^{N}(X,Y)\xi=X(f)Y-Y(f)X$, where $X,Y\in\Gamma(TN)$, we conclude that $\operatorname{grad}^{N} f=-c\,\xi$.
The Theorem \ref{th-submanifold} follows from (\ref{eq4.2}) and (\ref{eq4.3}).
\end{proof}

\begin{remark}
According to Theorem \ref{th-concurrent}, every generalized $\xi$-parallel submanifold with respect to the position vector field in $\mathbb{R}^n$ is minimal.
\end{remark}

Recall that a submanifold is said to have constant mean curvature (CMC) if the norm of its mean curvature vector field $H$ is constant, i.e., $|H|$ is constant. A submanifold $(M,g)$ of $(N,h)$ is called pseudo-umbilical if $A_H = |H|^2 I,$
where $I$ denotes the identity $(1,1)$-tensor field on $M$. According to Theorem \ref{th-submanifold}, we get the following Theorem.

\begin{theorem}\label{th-submanifold-m>2}
Let $(M,g)$ be an $m$-dimensional submanifold of a space form $(N,h)$ admitting a concircular vector field $\xi$ with nonzero concircular function $f$. Assume that $m>2$. If $(M,g)$ is a generalized $\xi$-parallel submanifold, then it is a CMC pseudo-umbilical biharmonic submanifold. Moreover, if the sectional curvature of $(N,h)$ is non-positive, then $(M,g)$ is minimal.
\end{theorem}

\begin{proof}
By using Theorem \ref{th-submanifold},  $f$ is constant along $M$, and
$H=(1/f)\operatorname{grad}^{N}f.$ Since $\operatorname{grad}^{N}f=-c\,\xi$, it follows that
$\tau(\mathbf{i})=mH=-(mc/f)\,\xi.$ Hence, by Remark \ref{remark}, $(M,g)$ is a biharmonic submanifold.
Moreover, from $H=-(c/f)\,\xi$, we obtain $\nabla^{N}_{X}H=-c\,X,$ for every tangent vector field $X$ on $M$. Consequently,
$(\nabla^{N}_{X}H)^\perp=0$ and $A_H(X)=c\,X.$ Therefore, $H$ is parallel in the normal bundle, and
$A_H=|H|^2I,$ since $|H|^2=c$. Thus, $(M,g)$ is a CMC pseudo-umbilical biharmonic submanifold.
Finally, if the sectional curvature $c$ of $(N,h)$ satisfies $c\le0$, then $(M,g)$ is minimal.
\end{proof}

\begin{example}
Let $\lambda(y)=\langle\alpha,y\rangle_{\mathbb{R}^{m+2}}$ be a smooth function on $\mathbb{S}^{m+1}$, where $\alpha=(\alpha_1,\ldots,\alpha_{m+2})\in\mathbb{R}^{m+2}$. Then,
$\xi=\operatorname{grad}^{\mathbb{S}^{m+1}}\lambda$ is a concircular vector field with concircular function $f=-\lambda$ \cite{YX}.
Consider the hypersurface
\[
M=\left\{y\in\mathbb{S}^{m+1}\; |\;
\alpha_1y_1+\cdots+\alpha_{m+2}y_{m+2}=\beta\right\},
\]
with $|\alpha|^2>\beta^2>0$. Since $\lambda=\beta$ on $M$, the function $f$ is constant along $M$, and the unit normal vector field of $M$ is given by
\[
\eta=\frac{1}{\sqrt{|\alpha|^2-\beta^2}}\,\xi.
\]
As $M$ is a totally umbilical hypersurface of $\mathbb{S}^{m+1}$, its
mean curvature vector is
\[
H=\frac{\beta}{|\alpha|^2-\beta^2}\,\xi.
\]
Hence, $fH=\operatorname{grad}^{N} f$ along $M$, if and only if $|\alpha|^2=2\beta^2.$
Therefore, by Theorem \ref{th-submanifold}, $(M,g)$ is a generalized $\xi$-parallel submanifold if and only if
$|\alpha|^2=2\beta^2.$ Furthermore, by Theorem \ref{th-submanifold-m>2}, if $m>2$, this hypersurface is also a proper biharmonic submanifold if $|\alpha|^2=2\beta^2$.
\end{example}

\begin{theorem}\label{th-submanifold-m=2}
Let $(M,g)$ be a surface immersed in a space form $(N(c),h)$ admitting a concircular vector field $\xi$ with nonzero concircular function $f$. Then $(M,g)$ is generalized $\xi$-parallel if and only if $fH-(\operatorname{grad}^{N}f)^\perp=0.$
In particular, if $(M,g)$ is minimal, then it is generalized $\xi$-parallel if and only if
$h(\operatorname{grad}^{N}f,\eta)=0$ along $M$, where $\eta$ is a unit normal vector field on $M$.
\end{theorem}

\begin{example}
Consider the surface $(M,g)$ in $\mathbb{S}^{3}$ parametrized by
\[
(x_1,x_2)\mapsto (x_1,x_2,\beta_1x_1+\beta_2x_2),
\]
where $\beta_1$ and $\beta_2$ are constants. A unit normal vector field is given by
\[
\eta=
\frac{1+x_1^2+x_2^2+(\beta_1x_1+\beta_2x_2)^2}
{2\sqrt{1+\beta_1^2+\beta_2^2}}
\left(-\beta_1,-\beta_2,1\right).
\]
Let
$\lambda(y)=\langle\alpha,y\rangle_{\mathbb{R}^{4}}$
be a smooth function on $\mathbb{S}^{3}$, where
$\alpha=(\alpha_1,\ldots,\alpha_4)\in\mathbb{R}^{4}.$
Then,
$\xi=\operatorname{grad}^{\mathbb{S}^{3}}\lambda$
is a concircular vector field with concircular function
$f=-\lambda.$
Moreover, the normal component of $\xi$ along $M$ is
\[
\xi^\perp
=
\frac{-\alpha_1\beta_1-\alpha_2\beta_2+\alpha_3}
{1+\beta_1^2+\beta_2^2}
\left(-\beta_1,-\beta_2,1\right).
\]
Since $(M,g)$ is a totally geodesic surface of $\mathbb{S}^{3}$, we have
$H=0.$
By the characterization of generalized $\xi$-parallel surfaces,
\[
fH-(\operatorname{grad}^{\mathbb{S}^{3}}f)^\perp=0,
\]
and since
$(\operatorname{grad}^{\mathbb{S}^{3}}f)^\perp=-\xi^\perp,$
the surface $(M,g)$ is generalized $\xi$-parallel if and only if
$\xi^\perp=0.$ Therefore,
$(M,g)$ is generalized $\xi$-parallel if and only if $\alpha_3=\alpha_1\beta_1+\alpha_2\beta_2.$
\end{example}

\begin{remark}
The previous examples show that the class of biharmonic hypersurfaces is strictly larger than the class of generalized $\xi$-parallel hypersurfaces. Moreover, Theorem \ref{th-submanifold-m>2} provides a constructive method for obtaining CMC pseudo-umbilical biharmonic submanifolds in space forms.
\end{remark}

\end{document}